\documentclass[reqno]{amsart}
\usepackage{etex}
\usepackage[a4paper,hmarginratio=1:1]{geometry}
\usepackage{amssymb,amsfonts,amsmath}
\usepackage{comment} 
\usepackage[all,arc]{xy}
\usepackage{enumerate}
\usepackage{mathrsfs,mathtools}
\usepackage{todonotes,booktabs}
\usepackage{stmaryrd}
\usepackage{marvosym}
\usepackage{graphicx}
\usepackage{pdflscape} 	

\usepackage{bbm}
\usepackage{tikz}
\usepackage{tikz-cd}
\usetikzlibrary{trees}
\usetikzlibrary[shapes]
\usetikzlibrary[arrows]
\usetikzlibrary{patterns}
\usetikzlibrary{fadings}
\usetikzlibrary{backgrounds}
\usetikzlibrary{decorations.pathreplacing}
\usetikzlibrary{decorations.pathmorphing}
\usetikzlibrary{positioning}
	
\def\biblio{\bibliography{duality}\bibliographystyle{alpha}}

\usepackage{xcolor} 
\usepackage{graphicx}

\usepackage[pagebackref]{hyperref}

\definecolor{dark-red}{rgb}{0.5,0.15,0.15}
\definecolor{dark-blue}{rgb}{0.15,0.15,0.6}
\definecolor{dark-green}{rgb}{0.15,0.6,0.15}
\hypersetup{
    colorlinks, linkcolor=dark-red,
    citecolor=dark-blue, urlcolor=dark-green
}

\renewcommand*{\backref}[1]{}
\renewcommand*{\backrefalt}[4]{%
  \ifcase #1 %
No citations.
  \or
(cit. on p. #2).%
  \else
(cit on pp. #2).%
  \fi%
}

\usepackage[nameinlink,capitalise,noabbrev]{cleveref}

\newtheorem{thm}{Theorem}[section]
\newtheorem{cor}[thm]{Corollary}
\newtheorem{prop}[thm]{Proposition}
\newtheorem{lem}[thm]{Lemma}

\theoremstyle{definition}
\newtheorem{defn}[thm]{Definition}

\theoremstyle{remark}
\newtheorem{rem}[thm]{Remark}
 
\makeatletter
\let\c@equation\c@thm
\makeatother
\numberwithin{equation}{section}

\Crefname{figure}{Figure}{Figures}
\Crefname{assu}{Assumption}{Assumptions}
\Crefname{lem}{Lemma}{Lemmas}
\Crefname{thm}{Theorem}{Theorems}
\Crefname{prop}{Proposition}{Propositions}

\DeclareMathOperator{\im}{Im}

\newcommand{\Q}{\mathbb{Q}}

\newcommand{\Z}{\mathbb{Z}}

\let\lim\relax

\DeclareMathOperator{\lim}{lim}

\newcommand{\F}{\mathbb{F}}

\newcommand{\K}{\mathbb{K}}

\DeclareMathOperator{\hofib}{hofib}

\newcommand{\Lf}[1]{\widehat{L}_{#1}^f}

\title{A Whitehead theorem for periodic homotopy groups}

\author{Tobias Barthel}
\thanks{TB was supported by the DNRF92 and the European Unions Horizon 2020 research and innovation programme under the Marie Sklodowska-Curie grant agreement No.~751794.}
\address{Department of Mathematical Sciences, University of Copenhagen, Universitetsparken 5, 2100 K{\o}benhavn {\O}, Denmark}
\email{tbarthel@math.ku.dk}

\author{Gijs Heuts}
\thanks{GH was supported by NWO grant 016.Veni.192.186.}
\address{Mathematical Institute, Utrecht University, Budapestlaan 6, 3584 CD  Utrecht, The Netherlands}
\email{g.s.k.s.heuts@uu.nl}

\author{Lennart Meier}
\address{Mathematical Institute, Utrecht University, Budapestlaan 6, 3584 CD  Utrecht, The Netherlands}
\email{f.l.m.meier@uu.nl}

\date{\today}

\begin{document}

\begin{abstract}
We show that $v_n$-periodic homotopy groups detect homotopy equivalences between simply-connected finite CW-complexes.
\end{abstract}

\maketitle

\setcounter{tocdepth}{1}
\tableofcontents
\def\biblio{}

\section{Introduction}

The classical Whitehead theorem \cite{whitehead} states that a map of CW-complexes which induces an isomorphism on homotopy groups (for any choice of basepoint) is a homotopy equivalence. Combined with the Hurewicz theorem, it implies that an integral homology equivalence of nilpotent CW-complexes is a homotopy equivalence. This paper concerns a variant of such results where one replaces homotopy groups and integral homology by algebraic invariants arising from chromatic homotopy theory. 

The first such invariants are the Morava $K$-theories $K(n)$. These also depend on a prime $p$, but (as is common) we suppress it from the notation. The following was proved, but not explicitly stated, by Bousfield \cite{bousfield_homology}; an alternative proof is given by Hopkins and Ravenel \cite{hopkinsravenel}, who show that suspension spectra are local with respect to the wedge of all Morava $K$-theories.

\begin{thm}[Bousfield, Hopkins--Ravenel]
Let $f$ be a map between nilpotent spaces such that $K(n)_*f$ is an isomorphism for all $n$ and $p$. Then $f$ is a weak homotopy equivalence.
\end{thm}

For convenience we assume a prime $p$ has been fixed and we work $p$-locally throughout. Our focus will be on the $v_n$-periodic homotopy groups of spaces. Recall that a pointed finite CW-complex $V$ is said to be \emph{of type $n$} if $K(i)_*V \cong K(i)_*$ for $i < n$, but $K(n)_*V$ is non-trivial. The periodicity theorem of Hopkins and Smith \cite{nilpotence2} implies that after suspending $V$ sufficiently many times, it admits a \emph{$v_n$ self-map} $v\colon \Sigma^d V \rightarrow V$. By definition, such a map induces an isomorphism on the groups $K(n)_*V$ and acts nilpotently on the (reduced) $K(m)$-homology groups for $m \neq n$. 

Now for $X$ a pointed space and $n \geq 1$, one defines its \emph{$v$-periodic homotopy groups} to be
\begin{equation*}
    v^{-1}\pi_*(X; V) := \mathbb{Z}[v^{\pm 1}] \otimes_{\mathbb{Z}[v]} \pi_*\mathrm{Map}_*(V, X).
\end{equation*}
Here $\mathbb{Z}[v]$ is a graded ring (with $v$ in degree $d$) and $\pi_*\mathrm{Map}_*(V, X)$ (with $* \geq 2$) is regarded as a graded abelian group on which $v$ acts by
\begin{equation*}
    \pi_*\mathrm{Map}_*(V, X) \xrightarrow{v^*} \pi_*\mathrm{Map}_*(\Sigma^d V, X) \cong \pi_{*+d}\mathrm{Map}_*(V, X).
\end{equation*}

\begin{defn}
\label{def:vnequiv}
A map $f$ of pointed spaces is a \emph{$v_n$-equivalence} if $v^{-1}\pi_*(f; V)$ is an isomorphism.
\end{defn}

This definition depends only on $n$. Indeed, to see that it is independent of the choice of type $n$ space $V$ and $v_n$ self-map $v$, one applies Theorem 7 (the thick subcategory theorem) and Corollary 3.7 (the asymptotic uniqueness of $v_n$ self-maps) of \cite{nilpotence2}. An equivalent way of phrasing the definition above is to say that $f$ is a $v_n$-equivalence if $\Phi_n(f)$ is an equivalence of spectra, with $\Phi_n$ denoting the $n$th Bousfield--Kuhn functor \cite{kuhn_telescopic}.

\begin{rem}
Fix a finite type $n$ \emph{spectrum} $V$ with a $v_n$ self-map $v$ and write $T(n)$ for the mapping telescope
\begin{equation*}
    T(n) := \mathrm{hocolim}(V \xrightarrow{v} \Sigma^{-d} V \xrightarrow{v} \Sigma^{-2d} V \xrightarrow{v} \cdots).
\end{equation*}
By definition a map $f$ of \emph{spectra} is a $v_n$-periodic equivalence if and only if $T(n)_*f$ is an isomorphism. (Although $T(n)$ depends on the choice of $V$, the associated notion of $T(n)_*$-homology isomorphism does not.) Moreover, if $T(n)_*f$ is an isomorphism, then $K(n)_*f$ is an isomorphism. The telescope conjecture asserts the converse, but is only known to hold for $n=0$ and $n=1$. The reader should be warned that the relation between $v_n$-periodic equivalences and $T(n)_*$-homology isomorphisms of \emph{spaces} is much more subtle. In particular, there is no direct implication in either direction. Bousfield gives detailed results relating $v_n$-periodic equivalences of spaces to the notion of \emph{virtual} $K(n)$-equivalences in \cite{bousfield_homotopicallocalizations}.
\end{rem}

We say a map $f\colon X \rightarrow Y$ is \emph{simple}\footnote{We warn the reader that this notion is unrelated to the notion of a simple map in geometric topology, where one demands that the inverse images of points are contractible.} if the homotopy fiber $F$ of $f$ is connected, has abelian fundamental group, and $\pi_1 X$ acts trivially on the homotopy groups of $F$. Thus, a space $Z$ is simple if and only if $Z \rightarrow *$ is a simple map, and the homotopy fiber of a simple map is a simple space. Also, any map between simply-connected spaces is simple.

In addition to \Cref{def:vnequiv}, we say a map $f$ of pointed nilpotent spaces is a $v_0$-equivalence (resp.~$p$-local equivalence) if it is a rational weak homotopy equivalence (resp.~weak equivalence after $\Z_{(p)}$-localization). The aim of this paper is to prove the following Whitehead theorem for periodic homotopy groups:

\begin{thm}
\label{thm:main}
Let $f$ be a simple map of pointed nilpotent finite $CW$-complexes. If $f$ is a $v_n$-equivalence for every $n \geq 0$, then $f$ is a $p$-local equivalence.
\end{thm}

The hypothesis that the spaces involved be finite complexes cannot simply be omitted, as the example $f\colon K(\Z/p, m) \rightarrow *$ for $m\ge 1$ shows. Indeed, the space in question has vanishing $v_n$-periodic homotopy groups for every $n \geq 0$. In fact, one can replace $K(\Z/p,m)$ by (the connected cover of) $\Omega^\infty E$ for any dissonant spectrum $E$ to obtain a larger class of counterexamples.

The plan of this paper is as follows. In \Cref{sec:moravaK} we give a brief review of the integral Morava $K$-theories and their values on Eilenberg--Mac Lane spaces. \Cref{sec:periodichomotopy} reviews some of Bousfield's results on the relation between $v_n$-periodic equivalences and homotopical localizations of the category of pointed spaces. Finally, in \Cref{sec:proof} we prove \Cref{thm:main}.

\subsection*{Acknowledgements}

We would like to thank Tyler Lawson, for helpful conversations concerning multiplications on integral Morava K-theory, as well as Tomer Schlank. We moreover thank the referee for catching an inaccuracy.

The authors would like to thank the Isaac Newton Institute for Mathematical Sciences for support and hospitality during the programme ``Homotopy harnessing higher structures'' when work on this paper was undertaken. This work was supported by EPSRC grant number EP/R014604/1.

\subsection*{Conventions}
Throughout this document, we fix a prime $p$ and all our spaces are assumed to be pointed, connected, nilpotent, and $p$-local. For the theory of arithmetic localizations of nilpotent spaces we refer to \cite[Chapter V]{bousfield_homotopy_1972}. In particular, the homotopy groups of the $R$-localization of $X$ for some subring $R\subset \mathbb{Q}$ will be just $\pi_*X\otimes R$ (with appropriate meaning for $\pi_1$) and similarly for homology. We understand a finite space to be the $p$-localization of a finite CW-complex. If $X$ is a pointed space we write $X\langle n \rangle$ for its $n$-connected cover and $\tau_{\leq n} X$ for its $n$th Postnikov section.

\section{Integral Morava $K$-theories}
\label{sec:moravaK}

In this section we will study the basic properties of integral Morava $K$-theories $\K(n)$. These are $2(p^n-1)$-periodic complex orientable integral lifts of the ordinary Morava $K$-theories $K(n)$ at height $n$ and prime $p$. The $p$-complete version of $\K(n)$ was first studied by Morava in \cite{morava_weil}; for the convenience of the reader, we briefly recall the construction. 

Let $MU$ be ($p$-local) complex cobordism with coefficients $\pi_*MU \cong \Z_{(p)}[t_1,t_2,\ldots]$, where $t_i$ has degree $2i$ and $t_0 := p$. 
Following \cite[Ch.~V]{ekmm}, we extend $t_i$ to a map $\Sigma^{2i}MU \to MU$ and denote  the corresponding cofiber by $M(i)$. 

\begin{defn}
We define an \emph{integral Morava $K$-theory} $\K(n)$ at height $n \ge 1$ and prime $p$ as 
\[
\K(n) := MU[t_{p^n-1}^{-1}] \otimes \bigotimes_{i \notin \{0,p^n-1\}}M(i).
\]
Here, the smash products are taken over $MU$.
\end{defn}

Note that the construction of $\K(n)$ depends on a choice of generators $t_1,t_2,\ldots,t_{n-1},t_n,\ldots$; we fix one such choice and consequently omit it from the notation. Furthermore, we shall write $v_n:=t_{p^n-1}$ from now on. The next lemma summarizes the salient features of the integral Morava $K$-theories.

\begin{lem}
The integral Morava $K$-theories have the following properties:
    \begin{enumerate}
        \item For any $p$ and $n$, the spectrum $\K(n)$ has the structure of an even complex orientable ring spectrum with graded commutative coefficient ring $\pi_*\K(n) \cong \Z_{(p)}[v_{n}^{\pm 1}]$, where $v_n$ is in degree $2(p^n-1)$. 
        \item Write $K(n)$ for $2(p^n-1)$-periodic Morava $K$-theory. There is a cofiber sequence
        \[
        \K(n) \xrightarrow{p} \K(n) \to K(n), 
        \]
        so in particular an equality of Bousfield classes $\langle \K(n) \rangle = \langle K(0) \oplus K(n) \rangle$.
        \item For any $p$ and $n$, $\K(n)$ has the structure of an $\mathbb{A}_{\infty}$-ring spectrum. If $p>2$, then the multiplication on $\K(n)$ is homotopy commutative. 
    \end{enumerate}
\end{lem}
\begin{proof}
By construction, $\K(n)$ is an even homotopy associative ring spectrum with the indicated coefficients, equipped with a homomorphism $MU \to \K(n)$ of ring spectra. Moreover, multiplication by $p$ on $\K(n)$ induces an equality of Bousfield classes
\[
\langle \K(n) \rangle = \langle \K(n)/p\rangle \oplus \langle p^{-1}\K(n) \rangle.
\]
Since $\K(n)/p \simeq K(n)$ and $\langle p^{-1}\K(n) \rangle = \langle K(0) \rangle$, Claim (2) follows.

It remains to prove (3). In \cite[Cor.~3.2]{angeltveit_thha} Angeltveit shows that a quotient $R/x$ of an even commutative $S$-algebra $R$ by a non-zero divisor $x$ admits an $\mathbb{A}_{\infty}$ $R$-algebra structure. Applying this to $R = MU$ and the spectra $M(i)$ defined above, it follows that $\K(n)$ admits the structure of an $\mathbb{A}_{\infty}$-ring spectrum. Finally, if $p>2$, $n\ge 1$, and $I = (t_1,t_2,\ldots,t_{p^n-2},t_{p^n},t_{p^n+1},\ldots)$, then the coefficients of $MU/I$ are concentrated in degrees divisible by $4 \le 2p^n-2 = |t_{p^n-1}| $. Theorem V.4.1 in \cite{ekmm} implies that $\K(n) = MU/I[t_{p^n-1}^{\pm 1}]$ is a homotopy commutative and associative ring spectrum, so $\K(n)$ is homotopy commutative whenever $p>2$.
\end{proof}

\begin{rem}
If $p=2$ and $n > 1$, then $\K(n)$ is not homotopy commutative. We have learned the following argument from Tyler Lawson; a more detailed account will be contained in forthcoming work by Lawson. Let $p=2$ and suppose for contradiction that the multiplication on $\K(n)$ constructed above is homotopy commutative. Consider then its homotopy commutative connective cover $R = \tau_{\ge 0}\K(n)$ and let $f\colon R \to H = H\F_2$ be the canonical quotient map. It would follow that the image of the induced map in homology $H_*R \to H_*H$ must be closed under the operation $Q_1$, which acts as $Q^{k+1}$ in degree $k$. However, a computation with the Tor spectral sequence for $H \otimes R \simeq (H\otimes BP) \otimes_{BP} R$ shows that this image is
\[
\F_2[\overline{\xi}_1^2, \overline{\xi}_2,\ldots, \overline{\xi}_{n-1}, \overline{\xi}_{n}^2, \overline{\xi}_{n+1},\ldots],
\]
which is not closed under $Q_1$ if $n >1$, because $Q_1(\overline{\xi}_{i}) = Q^{2^i}(\overline{\xi}_{i}) = \overline{\xi}_{i+1}$ for all $i\ge 1$ by \cite[Theorem III.2.2]{bmms86}. This gives the claim.
\end{rem}

Let $D_{p^{\infty}}$ be a $p$-divisible graded $\K(n)_*$-module concentrated in even degrees. Consider the graded commutative $\K(n)_*$-algebra 
\[
\Lambda(D_{p^{\infty}}) := \K(n)_* \oplus D_{p^{\infty}}[-1],
\]
obtained as the split square-zero extension of $\K(n)_*$ by a copy of $D_{p^{\infty}}$ shifted down one homological degree. In our examples, $D_{p^\infty}$ will be a direct sum of copies of $\mathbb{Q}/\mathbb{Z}_{(p)}$ in every even degree. 

\begin{prop}\label{prop:EilenbergMacLane}
Let $\pi$ be a finite abelian $p$-group, $n\ge 1$, and $K(\pi,d)$ the associated Eilenberg--Mac Lane space. Then there is an isomorphism 
\[
\K(n)_*(K(\pi,d)) \cong \Lambda(D_{p^{\infty}})
\]
of graded commutative $\K(n)_*$-algebras, for some $p$-divisible torsion graded even $\K(n)_*$-module $D_{p^{\infty}}$ depending on $\pi$, $n$ and $d$. Moreover, if $\pi$ is non-trivial, then $D_{p^{\infty}} = 0$ if and only if $n < d$.\footnote{In fact, the module $D_{p^{\infty}}$ can be described explicitly in terms of $\pi$, $n$, and $d$, but we will not need the precise formula here.}
\end{prop}
\begin{proof}
By the Ravenel--Wilson computation of the Morava $K$-homology of Eilenberg--Mac Lane spaces \cite{ravenelwilson_knem} \cite[Appendix]{johnsonwilson}, $K(n)_*K(\pi,d)$ is an even finitely generated module over $K(n)_*$, and the same papers establish the vanishing result claimed at the end of the proposition. Recall that, as a consequence of the classification of injective abelian groups (see e.g.\ \cite[Theorem 23.1]{fuchs}), a $p$-divisible $p$-torsion abelian group is uniquely determined by its $p^1$-torsion. Let $D_{p^{\infty}}$ be the uniquely determined even graded $p$-divisible torsion $\K(n)_*$-module whose submodule of $p^1$-torsion elements is isomorphic to the reduced homology $\overline{K(n)}_{*}(K(\pi,d))$, i.e.\ the injective hull of the latter graded group. 

Consider the long exact sequence in homotopy associated to the cofiber sequence
\[
\K(n) \otimes K(\pi,d) \xrightarrow{p} \K(n) \otimes K(\pi,d) \to K(n) \otimes K(\pi,d)
\]
By evenness, multiplication by $p$ on $\overline{\K(n)}_*(K(\pi,d))$ is injective in even degrees and surjective in odd degrees. Since $\Q \otimes \overline{\K(n)}_*(K(\pi,d)) = 0$, it thus follows that $\overline{\K(n)}_*(K(\pi,d))$ is trivial in even degrees and $p$-divisible in odd degrees. We thus get an isomorphism 
\[
\phi\colon \K(n)_*(K(\pi,d)) \cong \Lambda(D_{p^{\infty}})
\]
of $\K(n)_*$-modules. In any $\K(n)_*$-algebra structure on $\Lambda(D_{p^{\infty}})$ the product of two elements of $D_{p^{\infty}}$ must be zero as it is $p$-power torsion and in even degree. Thus, $\Lambda(D_{p^{\infty}})$ carries a unique $\K(n)_*$-algebra structure, namely that of a split square-zero extension, and hence $\phi$ must be an isomorphism of $\K(n)_*$-algebras. 
\end{proof}

Note that we have in particular shown that the ring $\K(n)_*(K(\pi,d))$ is graded commutative also in the case $p=2$, $n>1$, where $\K(n)$ is not a homotopy commutative ring spectrum. The divisible and hence non-finitely generated summands in $\K(n)_*(K(\pi,d))$ will play a crucial role in the proof of the main theorem.

\begin{rem}
We could also work with the functor $L_{K(0) \oplus K(n)}(E_n \otimes (-))$ in place of $\K(n)$ in this paper, where $E_n$ is a Morava $E$-theory of height $n$. The chromatic fracture square for $K(0) \oplus K(n)$ combined with Peterson's lift \cite{peterson_erw} of the Ravenel--Wilson computation to Morava $E$-theory can be used to establish the properties of this functor which we require for our applications. 
\end{rem}

\section{Periodic homotopy and homology equivalences}
\label{sec:periodichomotopy}

The aim of this section is to state and prove \Cref{thm:Lnperiodicequiv}, implying a relation between $v_n$-periodic homotopy groups and $E_n$-homology groups of spaces. We will rely on results concerning periodizations (or nullifications) of spaces as developed by Bousfield \cite{bousfield_periodicity,bousfield_telescopic} and Dror-Farjoun \cite{dror_localization}.

Following Section 4 of \cite{bousfield_telescopic} we begin by fixing a finite CW-complex $V_{n+1}$ of type $n+1$ that is also a suspension. For later reference we will denote by $d_{n+1}$ the lowest dimension in which $V_{n+1}$ has a nonvanishing homotopy group. There is an associated functor
\begin{equation*}
\Lf{n}\colon \mathcal{S}_* \rightarrow \mathcal{S}_*
\end{equation*}
which is the left Bousfield localization with respect to the map $V_{n+1} \rightarrow *$. This functor takes values in spaces $X$ which are $V_{n+1}$-null, meaning that the map $V_{n+1} \rightarrow *$ induces an equivalence
\begin{equation*}
    X \rightarrow \mathrm{Map}(V_{n+1}, X).
\end{equation*}
Also, for any $X$ there is a natural map $\eta_X\colon X \rightarrow \Lf{n} X$ which is homotopy initial for maps into $V_{n+1}$-null spaces. The behaviour of this functor with respect to periodic homotopy groups is as follows (cf.~\S4.6 of \cite{bousfield_telescopic}):
\begin{itemize}
    \item [(i)] The map $\eta_X$ induces an isomorphism of rational homotopy groups.
    \item[(ii)] The map $\eta_X$ induces an isomorphism of $v_i$-periodic homotopy groups for $1 \leq i \leq n$.
    \item[(iii)] The $v_i$-periodic homotopy groups of $\Lf{n} X$ vanish for $i > n$.
\end{itemize}

Let us call a map $\varphi$ an \emph{$\Lf{n}$-equivalence} precisely if $\Lf{n}(\varphi)$ is an equivalence of spaces. A crucial result is the following converse to the above:

\begin{thm}[Bousfield \cite{bousfield_telescopic}, Corollary 4.8]
\label{thm:Bousfield}
A map $\varphi\colon X \rightarrow Y$ of $d_{n+1}$-connected spaces is an $\Lf{n}$-equivalence if and only if it is a $v_i$-equivalence for every $0 \leq i \leq n$.
\end{thm}
We remark that some connectivity assumption is indeed necessary here. As by Ravenel--Wilson $K(n)_*K(\pi, n) \neq 0$ for a non-trivial finite abelian $p$-group $\pi$, we see that the $\Lf{n}$-localization can change under connected covers if the space is not at least $n$-connected. In contrast, $v_n$-periodic homotopy theory is insensitive to connected covers by the following well-known lemma.

\begin{lem}\label{lem:connectedcover}
Let $m, n\geq 1$ and $X$ be an arbitrary space. Then the canonical map $X\langle m \rangle \to X$ is a $v_n$-equivalence.
\end{lem}
\begin{proof}
With notation as in the introduction, $v^{-1}\pi_k(X;V)$ depends only on $\pi_{k+Nd}\mathrm{Map}(V,X)$ for $N$ large (and $d$ positive), which in turn only depends on $X\langle m \rangle$ for $m$ large. 
\end{proof}

The next lemma shows that the localization functor $\Lf{n}$ described above is compatible with its stable counterpart
\begin{equation*}
    L_n^f\colon \mathrm{Sp} \rightarrow \mathrm{Sp},
\end{equation*}
which localizes away from the finite type $n+1$ spectrum $\Sigma^\infty V_{n+1}$. By construction, $L_n^f$ is a finite and thus smashing localization functor, i.e., $L_n^fX \simeq X \otimes L_n^fS^0$ for any spectrum $X$.

\begin{lem}
\label{lem:Lnflocal}
If $E$ is an $L_n^f$-local spectrum and $\varphi$ is an $\Lf{n}$-equivalence of spaces, then $\varphi$ is also an $E_*$-equivalence.
\end{lem}
\begin{proof}
If $E$ is an $L_n^f$-local spectrum (i.e., $\overline{E}_*(V_{n+1}) = 0$), then $\Omega^\infty E$ is $V_{n+1}$-null and thus an $\Lf{n}$-local space. Hence, if $\varphi$ is an $\Lf{n}$-equivalence of spaces, then it follows for any $L_n^f$-local spectrum $E$ that $\mathrm{map}(\Sigma^{\infty}\varphi,E) \simeq \mathrm{Map}(\varphi,\Omega^{\infty}E)$ is an equivalence, so $\Sigma^\infty \varphi$ is an $L_n^f$-equivalence of spectra. Since $L_n^f$ is smashing, we get
\[
E \otimes \Sigma^{\infty}\varphi \simeq L_n^f(E) \otimes \Sigma^{\infty}\varphi \simeq E \otimes L_n^f\Sigma^{\infty}\varphi,
\]
which is an equivalence as observed above.
\end{proof}

We write $L_n$ for the Bousfield localization with respect to the $n$th Morava $E$-theory $E_n$. Equivalently, it is Bousfield localization with respect to the direct sum
\begin{equation*}
    K(0) \oplus K(1) \oplus \cdots \oplus K(n).
\end{equation*}
By \Cref{lem:Lnflocal}, any $\Lf{n}$-equivalence of spaces is in particular an $L_n$-equivalence. 

The point of this section is the following variation on (one direction of) \Cref{thm:Bousfield}:

\begin{prop}
\label{thm:Lnperiodicequiv}
Let $\varphi\colon X \rightarrow Y$ be an $(n+1)$-connected map (i.e., a map with $n$-connected homotopy fiber). If $\varphi$ is a $v_i$-equivalence for $0 \leq i \leq n$, then it is an $L_n$-equivalence.
\end{prop}

The proof of this proposition will use the following well-known lemma:

\begin{lem}
\label{lem:fibration}
Consider a spectrum $E$ and a diagram of spaces
\[
\begin{tikzcd}
F \ar{d}\ar{r} & F' \ar{d} \\
Y \ar{d}\ar{r} & Y' \ar{d} \\
X \ar{r} & X'
\end{tikzcd}
\]
in which the columns are fiber sequences. If the bottom map is an equivalence and the top map is an $E_*$-equivalence, then the middle map is also an $E_*$-equivalence.
\end{lem}
\begin{proof}
The induced map of Atiyah--Hirzebruch spectral sequences for these two fibrations gives an isomorphism of $E^2$-pages, from which the result follows immediately.
\end{proof}

\begin{proof}[Proof of \Cref{thm:Lnperiodicequiv}]
Write $F$ for the homotopy fiber of $\varphi$. By assumption it is an $n$-connected space with vanishing $v_i$-periodic homotopy groups for $i \leq n$. By \Cref{lem:fibration} it suffices to show that $L_n F$ is contractible. Consider the fiber sequence
\begin{equation*}
    F\langle d_{n+1} \rangle \rightarrow F \rightarrow \tau_{\leq d_{n+1}}F.
\end{equation*}
The space $F\langle d_{n+1} \rangle$ is $d_{n+1}$-connected and has vanishing $v_i$-periodic homotopy groups for $i \leq n$ by \Cref{lem:connectedcover}. Thus \Cref{thm:Bousfield} implies that $\Lf{n}(F\langle d_{n+1} \rangle)$ is contractible and \Cref{lem:Lnflocal} implies that the same is true for $L_n(F\langle d_{n+1} \rangle)$. Applying \Cref{lem:fibration} to the fiber sequence above, we see that it suffices to prove that $L_n \tau_{\leq d_{n+1}}F$ is contractible. Using \Cref{lem:fibration} together with the vanishing of $(E_n)_*K(\mathbb{Z}/p,k)$ for $k > n$ established by Ravenel and Wilson (cf.~\Cref{prop:EilenbergMacLane}), a finite induction on the Postnikov tower of $\tau_{\leq d_{n+1}}F$ now gives the desired conclusion. Note that we have used the fact that the homotopy groups of $F$ are torsion.
\end{proof}

\section{Proof of the main result}
\label{sec:proof}

Before we prove our Whitehead theorem in its general form, we will prove it for the special case of a map $X \to \ast$ to demonstrate the basic idea.

\begin{prop}\label{prop:absolutecase}
Let $X$ be a nilpotent finite space with abelian fundamental group such that $X \to \ast$ is a $v_n$-equivalence for all $n\geq 0$. Then $X$ is contractible.
\end{prop}
\begin{proof}
Note that the assumption on $X$ implies that all homotopy groups of $X$ are finite. We show by induction that $X$ is $n$-connected for every $n$, the base case being $n=0$. Assume thus that $X$ is already $(n-1)$-connected and consider the fiber sequence
\[X\langle n \rangle \to X \to K(\pi_nX, n).\]
As $K(\pi_nX, n)$ and $X$ have trivial $v_i$-periodic homotopy groups for $i\le n$, so does $X\langle n\rangle$. By \Cref{thm:Lnperiodicequiv} this implies that $L_n X\langle n \rangle$ is contractible and thus $\K(n)_*X \cong \K(n)_*(K(\pi_nX,n))$ by \Cref{lem:fibration}. As $X$ is a finite CW-complex, this must be a finitely generated $\K(n)_*$-module. By \Cref{prop:EilenbergMacLane} this is only possible if $\pi_nX =0$, implying that $X$ is $n$-connected. 
\end{proof}
\begin{rem}
Observe that the previous proposition does not need that $X$ is simple, only that it is nilpotent with abelian fundamental group. It is therefore slightly stronger than the specialization of \Cref{thm:main} to the absolute case. In fact with some care the hypothesis that $X $ be nilpotent can also be removed in \Cref{prop:absolutecase}.
\end{rem}



The remainder of this section is devoted to the proof of \Cref{thm:main}. We will set up an induction using the Moore--Postnikov tower of $f$. This is a tower of factorizations of $f$
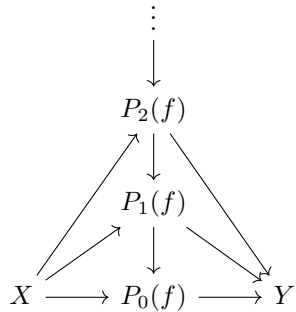
\[
\begin{tikzcd}
& \vdots \ar{d} & \\
& P_2(f) \ar{d} \ar{ddr} & \\
& P_1(f) \ar{d}\ar{dr} & \\
X \ar{r}\ar{ur}\ar{uur} & P_0(f) \ar{r} & Y
\end{tikzcd}
\]
characterized (up to equivalence) by the following properties:
\begin{itemize}
    \item[(1)] The map $X \rightarrow P_n(f)$ is $n$-connected, meaning it is an isomorphism on $\pi_i$ for $i < n$ and surjective on $\pi_n$.
    \item[(2)] The map $P_n(f) \rightarrow Y$ is an isomorphism on $\pi_i$ for $i > n$ and injective on $\pi_n$.
\end{itemize}
Note that this implies that $\pi_n(P_n(f)) = \im(\pi_n f)$. If $Y$ is contractible, the Moore--Postnikov tower reproduces the Postnikov tower of $X$; if $X$ is contractible, it gives the Whitehead tower of $Y$. Our assumption that $f$ is a simple map guarantees that the Moore--Postnikov tower is in fact a tower of principal fibrations, see \cite{may_simplicial}.

Given our assumption that $X$ and $Y$ are connected and $\pi_1 f$ is surjective (since the homotopy fiber of $f$ is connected), the map $P_1(f) \rightarrow Y$ is an equivalence. Our inductive hypothesis will be that $P_n(f) \rightarrow Y$ is an equivalence, from which we will deduce that $P_{n+1}(f) \rightarrow Y$ is an equivalence as well. Since 
\begin{equation*}
X \rightarrow \mathrm{holim}_n P_n(f)
\end{equation*}
is an equivalence, this proves the theorem.

We start with the following straightforward observation:

\begin{lem}
\label{lem:MP}
The maps $X \rightarrow P_n(f)$ and $P_n(f) \rightarrow Y$ are $v_i$-equivalences for all $n, i \geq 0$.
\end{lem}
\begin{proof}
For $i=0$ this is clear. Since $P_n(f) \rightarrow Y$ induces an equivalence on $n$-connected covers, it is a $v_i$-equivalence for any $i > 0$ using \Cref{lem:connectedcover}. The statement for $X \rightarrow P_n(f)$ now follows by two-out-of-three.
\end{proof}

\begin{lem}
The map $X \rightarrow P_{n+1}(f)$ is an $L_n$-equivalence.
\end{lem}
\begin{proof}
The map of the lemma is $(n+1)$-connected by construction and a $v_i$-equivalence for $i \geq 0$ by \Cref{lem:MP}. Therefore \Cref{thm:Lnperiodicequiv} gives the conclusion.
\end{proof}

\begin{cor}
\label{cor:Pnffinite}
If $K$ is an $E_n$-local ring spectrum, then $K_*(P_{n+1}(f))$ is a finitely generated $K_*$-module.
\end{cor}

\begin{lem}
\label{lem:fiberMPtower}
Let $n \geq 1$ and assume $P_n(f) \rightarrow Y$ is an equivalence. Then the fiber of $P_{n+1}(f) \rightarrow Y$ is a $K(\pi,n)$ for $\pi = \pi_n\hofib(f)$ a finite abelian $p$-group.
\end{lem}
\begin{proof}
Denote the fiber of $P_{n+1}(f) \rightarrow Y$ by $F$. The exact sequence
\[ \cdots \to \pi_{n+1}F \to \pi_{n+1}P_{n+1}(f) \hookrightarrow \pi_{n+1}Y \to \pi_nF \to \pi_nP_{n+1}(f) \to \pi_nY \to \cdots \]
receives a map from the exact sequence for the fiber sequence $\hofib(f) \to X \to Y$, inducing an isomorphism $\pi_n\hofib(f) \to \pi_nF$ by the five lemma. Moreover, the exact sequence above implies that in fact $F$ is a $K(\pi,n)$. To see this, use that $\pi_i P_{n+1}(f)\rightarrow \pi_i Y$ is bijective for $i > n+1$ and injective for $i = n+1$ by construction. Also $\pi_i P_{n+1}f = \pi_i X$ for $i \leq n$ and the assumption that $P_n(f) \rightarrow Y$ is an equivalence implies that $\pi_i X \rightarrow \pi_i Y$ is surjective for $i=n$ and bijective for $i<n$. Finally, $\pi$ is indeed finite as $f$ is a rational equivalence.
\end{proof}

We will exploit the non-finiteness of $\mathbb{K}(n)_*K(\pi,n)$. The crucial fact is the following:

\begin{lem}
\label{lem:EMfiber}
Consider a fiber sequence
\begin{equation*}
K(\pi, n) \rightarrow E \rightarrow Y    
\end{equation*}
with $n \geq 1$ and $Y$ finite such that $\pi$ is a nonzero finite abelian $p$-group and $\pi_1 E$ acts trivially on $\pi$. Then $\mathbb{K}(n)_*E$ is not finitely generated over $\mathbb{K}(n)_*$. 
\end{lem}
\begin{proof}
For the duration of this proof we will work in the $\mathbb{Z}/2$-graded setting, separating odd and even degrees. Recall from \Cref{prop:EilenbergMacLane} that $\mathbb{K}(n)_*K(\pi,n) = \Lambda(D_{p^{\infty}})$, where $\Lambda(D_{p^{\infty}})$ is the split square-zero extension of $\mathbb{Z}_{(p)}$ in even degree by a $p$-divisible graded $\K(n)_*$-module $D_{p^{\infty}}[-1]$ in odd degrees. Note that $D_{p^{\infty}}$ depends on $\pi$, but is nonzero.

We consider the Atiyah--Hirzebruch spectral sequence (AHSS)
\begin{equation*}
E^2_{s,t} \cong H_s(Y; \mathbb{K}(n)_t K(\pi,n)) \Rightarrow \mathbb{K}(n)_{s+t}E,
\end{equation*}
where $t\in\Z/(2p^n-2)$. Note that the $E^2$-page is given as stated because our assumptions imply that $\pi_1Y$ acts trivially on the homotopy groups of $K(\pi,n)$. 
As $\pi_1 E$ acts trivially on $\pi$, the fibration sequence of the lemma is in fact a principal fibration \cite[Lemma 4.70]{Hatcher}, i.e., can be extended to the right by a map $Y \rightarrow K(\pi,n+1)$. Thus, the AHSS becomes a spectral sequence of $\Lambda(D_{p^{\infty}})$-modules by \cite[\S15, Remark 4 on Page 352]{switzer}.\footnote{A more detailed proof will also appear in forthcoming work of Hedenlund, Krause, and Nikolaus.}

We write $M$ for the $E^2$-page viewed as a $\Z/2$-graded $\Lambda(D_{p^{\infty}})$-module, where we view $\Lambda(D_{p^{\infty}})$ as being $\Z/2$-graded by collecting the even and the odd degree elements, respectively. Moreover, we write $M_0 = E^2_{*,\text{even}}$ for the even part and $M_1 = E^2_{*,\text{odd}}$ for the odd one. 

Consider the Serre class $\mathcal{C}$ of finite abelian $p$-groups and denote by the same symbol the Serre class of $\Z/2$-graded $\Lambda(D_{p^{\infty}})$-modules whose underlying abelian groups are in $\mathcal{C}$. Then $M$ is a finitely generated free $\Lambda(D_{p^{\infty}})$-module mod $\mathcal{C}$. Indeed, this follows immediately from the observation that $M$ is a direct sum of the following kinds of groups:
\begin{itemize}
    \item[(1)] A finitely generated free $\Lambda(D_{p^{\infty}})$-module corresponding to the torsion-free summand of $H_*Y$.
    \item[(2)] A finite torsion group of the form $(H_*Y)_{\mathrm{tor}} \otimes \Lambda(D_{p^{\infty}})$.
    \item[(3)] A finite torsion group of the form $\mathrm{Tor}(H_*Y, \Lambda(D_{p^{\infty}}))$.
\end{itemize}
Also, the even part $M_0$ is finitely generated over $\mathbb{Z}_{(p)}$, whereas the odd part $M_1$ consists entirely of $p$-primary torsion. It immediately follows that the same is true of all subsequent pages of the AHSS.

For every $\mathbb{Z}/2$-graded $\Lambda(D_{p^{\infty}})$-module $N$, we have a map
\begin{equation}\label{eq:star}
\tag{$\star$}D_{p^{\infty}} \otimes N_0 \rightarrow N_1. 
\end{equation}
This map is an isomorphism mod $\mathcal{C}$ for $N = M$, simply because $M$ is free mod $\mathcal{C}$. We say more generally that a  graded $\Lambda(D_{p^{\infty}})$-module $N$ has property (P) whenever \eqref{eq:star} is an isomorphism mod $\mathcal{C}$. We now claim that every subsequent page of the AHSS has property (P). This will prove the lemma. Indeed, the AHSS collapses at a finite stage (since $H_*Y$ is bounded above), so that the $E^\infty$-page has property (P). The even part of the $E^\infty$-page contains a summand $\mathbb{Z}_{(p)}$ corresponding to the basepoint of $E$. Using property (P), it follows that the odd part of the $E^\infty$-page is not finitely generated over $\mathbb{Z}_{(p)}$, because $D_{p^{\infty}}$ is not.

To establish our claim we consider two cases.

\emph{Passing from $E^{2k}$ to $E^{2k+1}$.} Write $N = E^{2k}$ and consider it as a graded $\Lambda(D_{p^{\infty}})$-module. The differential $d_{2k}$ is of odd degree, i.e., it gives homomorphisms
\begin{equation*}
N_0 \xrightarrow{d_{2k}} N_1 \quad \text{and} \quad N_1 \xrightarrow{d_{2k}} N_0.    
\end{equation*}
Since $N_0$ is finitely generated over $\mathbb{Z}_{(p)}$ and $N_1$ is $p$-primary torsion, it follows that both these maps have image a finite abelian $p$-group (informally speaking, $d_{2k}$ is zero mod $\mathcal{C}$). It follows immediately that $E^{2k+1}$ is isomorphic to $E^{2k}$ mod $\mathcal{C}$. In particular, it still satisfies property (P).

\emph{Passing from $E^{2k-1}$ to $E^{2k}$.}
Write $N = E^{2k-1}$. This time the differential $d_{2k-1}$ is even. Since $d_{2k-1}$ is a $\Lambda(D_{p^{\infty}})$-module map, we obtain a commutative square
\[
\begin{tikzcd}
D_{p^{\infty}} \otimes N_0\ar{r} \ar{d}{d_{2k-1}} & N_1 \ar{d}{d_{2k-1}} \\
D_{p^{\infty}} \otimes N_0 \ar{r} & N_1.
\end{tikzcd}
\]
Since the horizontal maps are isomorphisms mod $\mathcal{C}$, it follows that the induced map on homology
\begin{equation*}
H(D_{p^{\infty}} \otimes N_0, d_{2k-1}) \rightarrow H(N_1, d_{2k-1})    
\end{equation*}
is an isomorphism mod $\mathcal{C}$. To conclude that $E^{2k}$ satisfies property (P), we need that the natural map
\begin{equation*}
D_{p^{\infty}} \otimes H(N_0, d_{2k-1})\rightarrow H(D_{p^{\infty}} \otimes N_0, d_{2k-1})
\end{equation*}
is an isomorphism mod $\mathcal{C}$. But this follows from the fact that the functor
\begin{equation*}
(\mathbb{Q}/\mathbb{Z}_{(p)}) \otimes - : \mathbf{Ab}^{\mathrm{fg}} \rightarrow \mathbf{Ab}
\end{equation*}
is exact mod $\mathcal{C}$ (in the target). Here $\mathbf{Ab}$ (resp.~$\mathbf{Ab}^{\mathrm{fg}}$) denotes the category of abelian groups (resp.~finitely generated abelian groups). This last statement is equivalent to the fact that $\mathrm{Tor}(\mathbb{Q}/\mathbb{Z}_{(p)}, A)$ is a finite abelian $p$-group whenever $A$ is finitely generated.
\end{proof}

We can now complete what we set out to do:
\begin{proof}[Proof of \Cref{thm:main}]
Combining \Cref{cor:Pnffinite} and \Cref{lem:EMfiber}, it follows that $\pi = 0$ with $\pi$ as in \Cref{lem:fiberMPtower}. Note that the application of \Cref{lem:EMfiber} uses the assumption on $f$ that $\pi_1 X$ acts trivially on the homotopy group $\pi_n\hofib(f)$. Then \Cref{lem:fiberMPtower} implies that $P_{n+1}(f) \rightarrow Y$ is an equivalence, which establishes the inductive step.
\end{proof}

\bibliographystyle{alpha}
\bibliography{references}
\end{document}